\documentclass[11pt]{amsart}

\usepackage[a4paper]{geometry}
\usepackage{latexsym, amsthm, amsfonts, amsmath,amssymb,graphicx,lmodern,mathrsfs}

\usepackage[latin1]{inputenc} 
\usepackage[T1]{fontenc} 
\usepackage[english]{babel} 

\newtheorem{theo}{Theorem}[section]
\newtheorem{lem}[theo]{Lemma}
\newtheorem{propo}[theo]{Proposition}
\newtheorem{coro}[theo]{Corollary}
\newtheorem*{theointro}{Theorem}

\theoremstyle{definition}
\newtheorem{defi}[theo]{Definition}

\theoremstyle{remark}
\newtheorem{rem}[theo]{Remark}
\newtheorem{ex}[theo]{Example}

\makeatletter

\@addtoreset{equation}{section}
\makeatother
\def\R{\mathbb{R}}
\def\Z{\mathbb{Z}}
\def\C{\mathbb{C}}
\def\N{\mathbb{N}}
\def\Q{\mathbb{Q}}

\def\z{\zeta}
\def\n{\eta}

\def\s{\sigma}
\def\a{\alpha}
\def\e{\varepsilon}

\def\b{\beta}

\def\n'{\nu}

\def\l{\lambda}

\def\k{\kappa}

\def\G{\Gamma}
\def\L{\Lambda}
\def\T{\Theta}

\def\sq {\sigma_{q}}

\begin{document}
\sloppy
\title[Meromorphic solutions and $q$-Borel-Laplace summation.]{Building meromorphic solutions of $q$-difference equations using a Borel-Laplace summation.}
\author{Thomas Dreyfus}
\address{Université Paris Diderot - Institut de Mathématiques de Jussieu,}
\curraddr{4, place Jussieu 75005 Paris.}
\email{thomas.dreyfus@imj-prg.fr.}
\thanks{Work partially supported by ANR, contract ANR-06-JCJC-0028.}



\date{\today}


\begin{abstract} 
After introducing $q$-analogs of the Borel and Laplace transformations, we prove that to every formal power series solution of a linear $q$-difference equation with rational coefficients, we may apply several $q$-Borel and Laplace transformations of convenient orders and convenient direction in order to construct a solution of the same equation that is meromorphic on $\C^{*}$. We use this theorem to construct explicitly an invertible matrix solution of a linear $q$-difference system with rational coefficients, of which entries are meromorphic on $\C^{*}$.~Moreover, when the system has two slopes and is put in the Birkhoff-Guenther normal form, we show how the solutions we compute are related to the one constructed by Ramis, Sauloy and Zhang.
 \end{abstract} 

\maketitle

\tableofcontents
\pagebreak[3]
\section*{Introduction}

Let us consider a formal power series $\hat{h}$ solution of a linear differential equation in coefficients in~$\C(z)$, where $z$ is a complex variable. Then, see \cite{B,R85,M95,VdPS}, we may apply to $\hat{h}$ several Borel and Laplace transformations of convenient orders and convenient direction, in order to obtain a solution that is meromorphic in the neighborhood of the origin on a sector of the Riemann surface of the logarithm, and that is asymptotic to $\hat{h}$. This summation process play an important role in differential Galois theory since, for instance, it is involved in the local analytic classification of linear differential equations. See \cite{VdPS}. \par The analytic and algebraic theory of $q$-difference equations has recently obtained many contributions in the spirit of Birkhoff program. See \cite{RSZ}. As in the differential case, the summation of formal power series solutions of linear $q$-difference equations in coefficients in~$\C(z)$ play a major role. The main goal of this paper is to explain how to transform such power series into solutions that are meromorphic on $\C^{*}$, using $q$-analogs of the Borel and Laplace transformations.\\
\begin{center}
$\ast\ast\ast$
\end{center}
 Throughout this paper, $q$ will be a fixed complex number with $|q|>1$. Let us consider as an example, the $q$-version of the Euler equation
$$z\sq y+y=z,$$
where $\sq$ is the dilatation operator that sends $y(z)$ to $y(qz)$. The latter equation  admits as a solution the formal power series with complex coefficients
$$\hat{h}:=\displaystyle\sum_{\ell\in \N}(-1)^{\ell}q^{\frac{\ell(\ell+1)}{2}}z^{\ell+1}.$$
To construct a meromorphic solution of the above equation, we use the Theta function 
$$\T_{q} (z):=\displaystyle \sum_{\ell \in \Z} q^{\frac{-\ell(\ell+1)}{2}}z^{\ell}=\displaystyle \prod_{\ell \in \N} (1-q^{-\ell-1})(1+q^{-\ell-1}z)(1+q^{-\ell}z^{-1}),$$ 
 which is analytic on $\C^{*}$, vanishes on the discrete $q$-spiral $-q^{\Z}:=\{-q^{n},n\in \Z\}$, with simple zero, and satisfies $$\s_{q}\T_{q}(z)=z\T_{q}(z)=\T_{q}\left(z^{-1}\right).$$
Then, for all $\l\in \left(\C^{*}/q^{\Z}\right)\setminus \{-1\}$ the following function $S_{q}^{[\l]}\left(\hat{h}\right)$, which is asymptotic to~$\hat{h}$, is solution of the same equation as $\hat{h}$ and is meromorphic on $\C^{*}$ with simple poles on the $q$-spiral $-\l q^{\Z}$:
$$S_{q}^{[\l]}\left(\hat{h}\right):=\displaystyle \sum_{\ell\in \Z}\frac{1}{1+q^{\ell}\l}\times\frac{1}{\T_{q}\left(\frac{q^{1+\ell}\l}{z}\right)}.$$
 
More generally, consider a formal power series solution of a linear $q$-difference equation in coefficients in $\C(z)$. Although there are several $q$-analogs of the Borel and Laplace transformations, see \cite{Abdi60,Abdi64,D3,DVZ,MZ,R92,RZ,Z99,Z00,Z01,Z02,Z03}, and we know the existence of a solution of the same equation that is meromorphic on $\C^{*}$, see \cite{Pra}, we did not know until this paper how to compute in general such solution using $q$-analogs of the Borel and Laplace transformations. \par 
Let $(\mu,K)\in \Q_{>0}\times \N^{*}$ with $K\in \mu\,\N^{*}$, and $\l\in \C^{*}/q^{K^{-1}\Z}$. Following \cite{R92,Z02}, we define the $q$-analogs of the Borel and Laplace transformations we will use, see Definition~\ref{defi1} for more precisions:
$$\begin{array}{ll}
\hat{\mathcal{B}}_{\mu}:\displaystyle\sum_{\ell\in \N} a_{\ell}z^{\ell}\longmapsto\displaystyle\sum_{\ell\in \N} a_{\ell}q^{\frac{-\ell(\ell-1)}{2\mu}} \z^{\ell},&\mathcal{L}_{\mu,K}^{[\l]}:f\longmapsto\dfrac{\mu}{K}
\displaystyle \sum_{\ell\in K^{-1}\Z}\frac{f\left(q^{\ell}\l\right)}{\T_{q^{1/\mu}}\left(\frac{q^{\frac{1}{\mu}+\ell}\l}{z}\right)}.
\end{array}
$$

We now present our main result, see Theorem \ref{theo1} and Proposition \ref{propo3} for a more precise statement.\\
\begin{theointro}
\textit{Let $\hat{h}$ be a formal power series solution of a linear $q$-difference equation in coefficients in $\C(z)$. There exist $\k_{1},\dots,\k_{r}\in \Q_{>0}$, $n,K\in \N^{*}$ and a finite set ${\Sigma\subset \C^{*}/q^{n^{-1}\Z}}$, we may compute from the $q$-difference equation, such that for all $\l\in  \left(\C^{*}/q^{n^{-1}\Z}\right)\setminus \Sigma$,
$$S_{q}^{[\l]}\left(\hat{h}\right):=\mathcal{L}_{\k_{r},n}^{[\l]}\circ\mathcal{L}_{\k_{r-1},K}^{[\l]}\circ\dots\circ \mathcal{L}_{\k_{1},K}^{[\l]}\circ\mathcal{\hat{B}}_{\k_{1}}\circ \dots\circ\mathcal{\hat{B}}_{\k_{r}}\left(\hat{h}\right),$$
is meromorphic on $\C^{*}$, and is solution of the same equation as $\hat{h}$. Moreover, $S_{q}^{[\l]}\left(\hat{h}\right)$ is asymptotic to $\hat{h}$ and, for $|z|$ close to $0$ it has pole of order at most $1$ that are contained in the $q^{1/n}$-spiral $-\l q^{n^{-1}\Z}$.}
\end{theointro} 

In $\S \ref{sec2}$, we give two applications of our main result. Consider a linear $q$-difference system of the form
$$
\sq Y=BY,
$$
where $B\in \mathrm{GL}_{m}\Big(\C(z)\Big)$, that is an invertible matrix with entries in $\C(z)$. Praagman in \cite{Pra} has proved that the above equation admits a fundamental solution, that is an invertible matrix solution, of which entries are meromorphic on $\C^{*}$. The proof is based on purely theoretical argument and is not constructive. Using the formal local classification of $q$-difference systems of \cite{VdPR}, and our main result,  we show how to construct a fundamental solution of $\sq Y=BY$, of which entries are meromorphic on $\C^{*}$. See Corollary \ref{coro1}.\\\par 
On the other side, the meromorphic solutions play a major role in Galois theory of $q$-difference equations, see \cite{Bu,BuT,DVRSZ,Ro08,RS07,RS09,RSZ,S00,S03,S04b,S04}. In particular the meromorphic solutions are used to prove the descent of the Galois group to the field $\C$, instead of the field of functions invariant under the action of~$\sq$, that is, the field of meromorphic functions on the torus $\C^{*}\setminus q^{\Z}$. Note that this latter field can be identified with the field of elliptic functions. In \cite{RSZ} the meromorphic solutions are obtained via successive gauge transformation. We show in a particular case, see Theorem~\ref{theo3}, how the solutions we compute with $q$-Borel and $q$-Laplace transformations are related to the meromorphic solutions appearing in \cite{RSZ}.\\\par 
\textbf{Acknowledgments.}
The author would like to thank Changgui Zhang, Jacques Sauloy, and the anonymous referees, for their suggestions to improve the quality of the paper. 

\pagebreak[3]
\section{Meromorphic solutions of linear $q$-difference equations.}\label{sec1}
 Let $\C[[z]]$ be the ring of formal power series and let $\hat{h}\in\C[[z]]$ be a solution of a linear $q$-difference equation in coefficients in $\C(z)$. The formal series $\hat{h}$ might diverges but we will see that we can construct from $\hat{h}$, a solution of the same equation that is meromorphic on~$\C^{*}$, and  that is asymptotic to $\hat{h}$. In $\S \ref{sec11}$, we define $q$-analogs of Borel and Laplace transformations. In $\S \ref{sec12}$ we prove that we might apply several $q$-Borel and $q$-Laplace transformations to~$\hat{h}$, to construct a solution  of the same linear $q$-difference equation that is meromorphic on $\C^{*}$, and  that is asymptotic to $\hat{h}$.
\pagebreak[3]
\subsection{Definition of $q$-analogs of Borel and Laplace transformations.}\label{sec11}
We start with the definition of the $q$-Borel and the $q$-Laplace transformations. Note that the $q$-Borel transformation was originally introduced in \cite{R92}. When $q>1$ is real and $\mu=K=1$, we recover the $q$-Laplace transformation and the functional space introduced in \cite{Z02}, Theorem 1.2.1. Earlier introductions of $q$-Laplace transformations can be found in \cite{Abdi60,Abdi64}. Those latter behave differently than our $q$-Laplace transformation, since they involve $q$-deformations of the exponential, instead of the Theta function. \\ \par 

Throughout the paper, we will say that two analytic functions are equal if their analytic continuation coincide. We fix, once for all a determination of the logarithm over the Riemann surface of the logarithm we call $\log$. If $\a\in \C$, then, we write $q^{\a}$ instead of~$e^{\a\log (q)}$. One has $q^{\a+\b}=q^{\a}q^{\b}$, for all $\a,\b\in \C$.  

\pagebreak[5]
\begin{defi}\label{defi1}
\begin{trivlist}
Let $(\mu,K)\in \Q_{>0}\times \N^{*}$ with $K\in \mu\,\N^{*}$, and $\l \in \C^{*}/q^{K^{-1}\Z}$.
\item (1) We define the $q$-Borel transformation of order $\mu$ as follows
$$
\begin{array}{llll}
\hat{\mathcal{B}}_{\mu}:&\C[[z]]&\longrightarrow&\C[[\z]]\\
&\displaystyle\sum_{\ell\in \N} a_{\ell}z^{\ell}&\longmapsto&\displaystyle\sum_{\ell\in \N} a_{\ell}q^{\frac{-\ell(\ell-1)}{2\mu}} \z^{\ell}.
\end{array}
$$
\item (2) Let  $\mathcal{M}(\C^{*},0)$ be the field of functions that are meromorphic on some punctured neighborhood of $0$ in~$\C^{*}$. An element $f$ of $\mathcal{M}(\C^{*},0)$ is said to belongs to $\mathbb{H}_{\mu,K}^{[\l]}$ if there exist~$\e>0$ and a connected domain $\Omega \subset \C$, such that:
\begin{itemize}
\item  $\displaystyle\bigcup_{\ell\in K^{-1}\Z}\Big\{z\in \C^{*}\Big| \left|z-\l q^{\ell}\right|<\e\left| q^{\ell}\l\right| \Big\}\subset \Omega.$
\item The function $f$ can be continued to an analytic function on $\Omega$ with $q^{1/\mu}$-exponential growth of order $1$ at infinity, which means that there exist constants $L,M>0$, such that for all $z\in \Omega$:
$$|f(z)|<L\Big|\T_{|q|^{1/\mu}}(M|z|)\Big| .$$
\end{itemize}
An element $f$ of $\mathcal{M}(\C^{*},0)$ is said to belongs to $\mathbb{H}_{\mu,K}$, if there exists a finite set~${\Sigma \subset \C^{*}/q^{K^{-1}\Z}}$, such that for all $\l \in \left(\C^{*}/q^{K^{-1}\Z}\right)\setminus \Sigma $, we have $f\in \mathbb{H}_{\mu,K}^{[\l]}$.
\item (3)   Because of Lemma \ref{lem2} below, the following map is defined and is called $q$-Laplace transformation of order $\mu$: 

$$
\begin{array}{llll}
\mathcal{L}_{\mu,K}^{[\l]}:&\mathbb{H}_{\mu,K}^{[\l]}&\longrightarrow&\mathcal{M}(\C^{*},0)\\
&f&\longmapsto&\dfrac{\mu}{K}
\displaystyle \sum_{\ell\in K^{-1}\Z}\frac{f\left(q^{\ell}\l\right)}{\T_{q^{1/\mu}}\left(\frac{q^{\frac{1}{\mu}+\ell}\l}{z}\right)}.
\end{array}
$$

For $|z|$ close to $0$, the function $\mathcal{L}_{\mu,K}^{[\l]}\left(f\right)(z)$ has poles of order at most $1$ that are contained in the~$q^{1/K}$-spiral $-q^{K^{-1}\Z}\l$. 
\end{trivlist}
\end{defi} 

\pagebreak[3]
\begin{rem}
The necessity of the factor $\frac{\mu}{K}$ in $\mathcal{L}_{\mu,K}^{[\l]}$ will appear clearly in Remark \ref{rem4}. 
\end{rem}

The following lemma generalizes \cite{Z02}, Lemma 1.3.1,  when $q$ is not real. Since the proof is basically the same, it will be only sketched.

\pagebreak[3]
\begin{lem}\label{lem2}
Let $(\mu,K)\in \Q_{>0}\times \N^{*}$ with $K\in \mu\,\N^{*}$, $\e>0$ and let us define $$\Omega:=\displaystyle\bigcap_{\ell\in K^{-1}\Z}\Big\{z\in \C^{*}\Big| \left|z+ q^{\ell}\right|\geq\e \left|q^{\ell} \right|\Big\}.$$
There exists $C>0$ such that we have on the domain $\Omega$:
$$\left| \T_{q^{1/\mu}}(z)\right|\geq C\left|\T_{|q|^{1/\mu}}(|z|)\right|.$$
\end{lem}

\begin{proof}
  Since the function $\left| \frac{\T_{q^{1/\mu}}(z)}{\T_{|q|^{1/\mu}}(|z|)}\right|$ is invariant under the action of ${\sq^{1/\mu}:=\sigma_{q^{1/\mu}}}$, we just have to prove that we have the inequality for $$z\in \displaystyle\Omega\bigcap \Big\{z\in \C^{*}, |z|\in \left[1,|q|^{1/\mu}\right]\Big\}=:\G.$$
We remind that $\T_{q^{1/\mu}}(z)$ is analytic on $\C^{*}$ and vanishes on the discrete $q^{1/\mu}$-spiral $-q^{\Z/\mu}$. Therefore, the function $f:=\left| \frac{\T_{q^{1/\mu}}(z)}{\T_{|q|^{1/\mu}}(|z|)}\right|$ is continuous and does not vanish on~$\G$. Since $\G$ is compact, $f$ admits a minimum $C>0$ on $\G$. This yields the result.
\end{proof}

The following lemma will be needed in $\S \ref{sec12}$.
\pagebreak[3]
\begin{lem}\label{lem1}
Let $(\mu,K)\in \Q_{>0}\times \N^{*}$ with $K\in \mu\,\N^{*}$, $\hat{h}\in \C[[z]]$, $\l\in \C^{*}/q^{K^{-1}\Z}$ and $g\in  \mathbb{H}_{\mu,K}^{[\l]}$. Then \\
\begin{itemize}
\item $\hat{\mathcal{B}}_{\mu}\left(\sq^{1/K}\hat{h}\right)= \sq^{1/K}\hat{\mathcal{B}}_{\mu}\left( \hat{h}\right)$, where  ${\sq^{1/K}:=\sigma_{q^{1/K}}}$.\\
\item $\hat{\mathcal{B}}_{\mu}\left( z\sq^{1/\mu}\hat{h}\right)= \z\hat{\mathcal{B}}_{\mu}\left( \hat{h}\right)$.\\
\item $\mathcal{L}_{\mu,K}^{[\l]}\left(\sq^{1/K} g\right)= \sq^{1/K}\mathcal{L}_{\mu,K}^{[\l]}\left( g\right)$. \\
\item $\mathcal{L}_{\mu,K}^{[\l]}\left( \z g\right)= z\sq^{1/\mu}\mathcal{L}_{\mu,K}^{[\l]}\left( g\right)$.
\end{itemize}
\end{lem}

\begin{proof}
The two first equalities are straightforward computations. Let us prove the third equality. Since $K\in \mu\,\N^{*}$, we obtain $K^{-1}\Z+\mu^{-1}=K^{-1}\Z$. 
Then,
$$\begin{array}{lll}
\mathcal{L}_{\mu,K}^{[\l]}\left( \sq^{1/K}g\right)&=&\dfrac{\mu}{K}
\displaystyle \sum_{\ell\in K^{-1}\Z}\frac{g\left(q^{\frac{1}{K}+\ell}\l\right)}{\T_{q^{1/\mu}}\left(\frac{q^{\frac{1}{\mu}+\ell}\l}{z}\right)}\\
&=&\dfrac{\mu}{K}
\displaystyle \sum_{\ell\in K^{-1}\Z}\frac{g\left(q^{\ell}\l\right)}{\T_{q^{1/\mu}}\left(\frac{q^{\frac{1}{\mu}+\ell}\l}{q^{1/K} z}\right)} =\sq^{1/K}\mathcal{L}_{\mu,K}^{[\l]}\left(  g\right).
\end{array} $$

 Let us now prove the last equality.
$$\begin{array}{lll}
\mathcal{L}_{\mu,K}^{[\l]}\left( \z g\right)&=&\dfrac{\mu}{K}
\displaystyle \sum_{\ell\in K^{-1}\Z}\frac{q^{\ell}\l g\left(q^{\ell}\l\right)}{\T_{q^{1/\mu}}\left(\frac{q^{\frac{1}{\mu}+\ell}\l}{z}\right)}\\
&=&z \dfrac{\mu}{K}
\displaystyle \sum_{\ell\in K^{-1}\Z}\frac{g\left(q^{\ell}\l\right)}{\T_{q^{1/\mu}}\left(\frac{q^{\ell}\l}{z}\right)}=z\sq^{1/\mu}\mathcal{L}_{\mu,K}^{[\l]}\left( g\right).
\end{array} $$

\end{proof}

\pagebreak[3]
\begin{rem}\label{rem1}
 Let us keep the same notations as in Lemma \ref{lem1} and assume that ${\hat{\mathcal{B}}_{\mu}\left(\hat{h}\right)\in  \mathbb{H}_{\mu,K}^{[\l]}}$.  Using Lemma \ref{lem1}, it follows that
$$\sq^{1/K}\left(\mathcal{L}_{\mu,K}^{[\l]}\circ\hat{\mathcal{B}}_{\mu}\left( \hat{h}\right)\right)=\mathcal{L}_{\mu,K}^{[\l]}\circ\hat{\mathcal{B}}_{\mu}\left(\sq^{1/K} \hat{h}\right) 
\hbox{ and }
z\mathcal{L}_{\mu,K}^{[\l]}\circ\hat{\mathcal{B}}_{\mu}\left( \hat{h}\right)=\mathcal{L}_{\mu,K}^{[\l]}\circ\hat{\mathcal{B}}_{\mu}\left(z\hat{h}\right) .
$$
In particular, if we additionally assume that $\hat{h}\in \C[[z]]$  is solution of a linear $q$-difference equation in coefficients in $\C[z]$, then $\mathcal{L}_{\mu,K}^{[\l]}\circ\hat{\mathcal{B}}_{\mu}\left( \hat{h}\right)\in\mathcal{M}(\C^{*},0)$ is solution of the same equation as~$\hat{h}$. Since the latter solution belongs to $\mathcal{M}(\C^{*},0)$, we obtain automatically that it belongs to $\mathcal{M}(\C^{*})$, the field of functions that are meromorphic on~$\C^{*}$.\par 
More generally, if $\hat{h}\in \C[[z]]$  is solution of a linear $q$-difference equation in coefficients in~$\C[z]$, we wonder if there exists~$(\mu,K)\in \Q_{>0}\times \N^{*}$ with $K\in \mu\,\N^{*}$, so that we have~${\hat{\mathcal{B}}_{\mu}\left(\hat{h}\right)\in  \mathbb{H}_{\mu,K}}$. 
Unfortunately the answer is no as shown the following example and we will have to apply successively several $q$-Borel and $q$-Laplace transformations to obtain a meromorphic solution. See Theorem~\ref{theo1}. 
\end{rem}

\pagebreak[3]
\begin{ex}
 Let us consider the formal power series ${\hat{h}:=\left(\displaystyle\sum_{\ell\in \N}q^{\frac{\ell(\ell-1)}{2}}z^{\ell}\right)^{2}\in \C[[z]]}$, which is solution of the linear $q$-difference equation 
$$\left(q^{2}z^{3}\sq^{2}-z(z+1)\sq+1\right)\hat{h}=1+z. $$
We claim that for all $(\mu,K)\in \Q_{>0}\times \N^{*}$ with $K\in \mu\,\N^{*}$, we have  ${\hat{\mathcal{B}}_{\mu}\left(\hat{h}\right)\notin  \mathbb{H}_{\mu,K}}$.
Note that this series was already used in \cite{MZ},~Page~1872, as a similar counter-example. 
In the latter paper, it is shown that the series $\hat{\mathcal{B}}_{\mu}\left(\hat{h}\right)$ belongs to $\C\{\z\}$, the ring of germ of analytic functions at $0$, if and only if $\mu\leq 1$. Using Lemma \ref{lem1}, we find that for all $j\in \N$, $i\in \Q$, $\mu\in \Q_{>0}$ and $\hat{f}\in \C[[z]]$,
\begin{equation}\label{eq2}
\hat{\mathcal{B}}_{\mu}\left(z^{j}\sq^{i}\hat{f}\right)=\frac{\z^{j}\left(\sq\right)^{i-\frac{j}{\mu}}\hat{\mathcal{B}}_{\mu}\left(\hat{f}\right)}{q^{\frac{j(j-1)}{2\mu}}}.
\end{equation} 
Let $\mu \leq 1$. Following (\ref{eq2}), we obtain that $\hat{\mathcal{B}}_{\mu}\left(\hat{h}\right)$ is solution of 
$$\left(q^{2-3/\mu}\z^{3}\sq^{2-3/\mu}-q^{-1/\mu}\z^{2}\sq^{1-2/\mu}-\z\sq^{1-1/\mu}+1\right)\left(\hat{\mathcal{B}}_{\mu}\left(\hat{h}\right)\right)=1+\z. $$
Let $M>0$. Using the $q$-difference equation satisfied by $\T_{|q|^{1/\mu}}(M|\z|)$, we find that $$f:=\frac{\hat{\mathcal{B}}_{\mu}\left(\hat{h}\right)}{\T_{|q|^{1/\mu}}(M|\z|)}$$ is solution of 
\begin{equation}\label{eq7}
\begin{array}{l}
q^{2-3/\mu}M^{2\mu-3}|\z|^{2\mu}\left(\frac{\z}{|\z|}\right)^{3}\sq^{2-3/\mu}f-q^{-1/\mu}M^{\mu-2}|\z|^{\mu}\left(\frac{\z}{|\z|}\right)^{2}\sq^{1-2/\mu}f\\
-M^{\mu-1}|\z|^{\mu}\left(\frac{\z}{|\z|}\right)\sq^{1-1/\mu}f+f=\frac{1+\z}{\T_{|q|^{1/\mu}}(M|\z|)}.
\end{array}\end{equation}
Since $\mu\leq 1$, (\ref{eq7}) yields that $\frac{\hat{\mathcal{B}}_{\mu}\left(\hat{h}\right)}{\T_{|q|^{1/\mu}}(M|\z|)}$ can be continued to a meromorphic function on~$\C^{*}$ with no poles, if $\mu<1$, and with poles of order $1$ in the $q$-spiral $q^{\N}:=\{q^{n},n\in \N\}$ if $\mu=1$.
Moreover, for all $K\in \mu\,\N^{*}\bigcap\N^{*}$, $\l\in \C^{*}/q^{K^{-1}\Z}$, $L>0$, there exists $\ell \in \Z$,  such that 
$$\left|\frac{\hat{\mathcal{B}}_{\mu}\left(\hat{h}\right)\left(q^{\ell/K}\l\right)}{\T_{|q|^{1/\mu}}(M|q^{\ell/K}\l|)}\right|>L.$$
Here, we have made the convention, that if the meromorphic  function $f$ has a pole in~$\z_{0}$, then $|f(\z_{0})|=+\infty$. Hence, for all $K\in \mu\,\N^{*}\bigcap\N^{*}$ and for all $\l\in \C^{*}/q^{K^{-1}\Z}$, we have
 ${\hat{\mathcal{B}}_{\mu}\left(\hat{h}\right)\notin \mathbb{H}_{\mu,K}^{[\l]}}$. To conclude,  for all $(\mu,K)\in \Q_{>0}\times \N^{*}$ with $K\in \mu\,\N^{*}$, we find ${\hat{\mathcal{B}}_{\mu}\left(\hat{h}\right)\notin  \mathbb{H}_{\mu,K}}$. This proves our claim.
\end{ex}

\pagebreak[3]
\begin{rem}\label{rem4}
 Using the expression of $\T_{q}$, we find that for all $k\in \Z$, and all $\mu\in \Q_{>0}$, $$\T_{q^{1/\mu}}\left(q^{k/\mu}z\right)=q^{\frac{k(k-1)}{2\mu}}z^{k}\T_{q^{1/\mu}}(z).$$ 
Following the definition of the $q$-Laplace transformation, we obtain that, for all ${(\mu,K)\in \Q_{>0}\times \N^{*}}$ with $K\in \mu\,\N^{*}$ and $\l\in \C^{*}/q^{K^{-1}\Z}$,  
$$\mathcal{L}_{\mu,K}^{[\l]}\left(1\right)=1.$$
Using additionally Remark \ref{rem1}, it follows that if $\ell\in \N$, then for all $(\mu,K)\in \Q_{>0}\times \N^{*}$ with $K\in \mu\,\N^{*}$ and $\l\in \C^{*}/q^{K^{-1}\Z}$, the function $\hat{\mathcal{B}}_{\mu}\left(z^{\ell}\right)$ belongs to $\mathbb{H}_{\mu,K}^{[\l]}$ and
$$\mathcal{L}_{\mu,K}^{[\l]}\circ\hat{\mathcal{B}}_{\mu}\left(z^{\ell}\right)=z^{\ell}.$$
More generally, if $f$ belongs to $\C\{z\}$, then for all $(\mu,K)\in \Q_{>0}\times \N^{*}$ with $K\in \mu\,\N^{*}$ and ${\l\in \C^{*}/q^{K^{-1}\Z}}$, the function
$\hat{\mathcal{B}}_{\mu}\left(f\right)$ belongs to $\mathbb{H}_{\mu,K}^{[\l]}$ and
$$\mathcal{L}_{\mu,K}^{[\l]}\circ\hat{\mathcal{B}}_{\mu}\left(f\right)=f.$$
\end{rem}

We finish the subsection by describing an asymptotic property of the $q$-Laplace transformation. The following definition is inspirited by \cite{Z02}, Definition~1.3.2.

\pagebreak[3]
\begin{defi}
Let us consider$(\mu,K)\in \Q_{>0}\times \N^{*}$ with $K\in \mu\,\N^{*}$, $\l\in \C^{*}/q^{K^{-1}\Z}$, $f\in\mathcal{M}(\C^{*},0)$ and ${\hat{f}:=\displaystyle\sum_{i\in \N}\hat{f}_{i}z^{i}\in \C[[z]]}$. We say that $$f\displaystyle\sim^{[\l]}_{\mu,K}\hat{f},$$ if for all $\e,R>0$ sufficiently small, there exist $L,M>0$, such that for all $k\in \N$ and for all 
$z$ in $$\Big\{z\in \C^{*}\Big| |z|<R\Big\}\setminus \displaystyle\bigcup_{\ell\in K^{-1}\Z}\Big\{z\in \C^{*}\Big| \left|z+ q^{\ell}\l\right|<\e\left| q^{\ell}\l \right|\Big\},$$
we have 
$$\Bigg|f(z)-\displaystyle\sum_{i=0}^{k-1}\hat{f}_{i}z^{i} \Bigg|<LM^{k}|q|^{\frac{k(k-1)}{2\mu}}|z|^{k}.$$
\end{defi}
For $\mu\in \Q_{>0}$, we define the formal $q$-Laplace transformation of order $\mu$ as follows:
$$
\begin{array}{llll}
\hat{\mathcal{L}}_{\mu}:&\C[[\z]]&\longrightarrow&\C[[z]]\\
&\displaystyle\sum_{\ell\in \N} a_{\ell}\z^{\ell}&\longmapsto&\displaystyle\sum_{\ell\in \N} a_{\ell}q^{\frac{\ell(\ell-1)}{2\mu}} z^{\ell}.
\end{array}
$$
Using Remark \ref{rem4}, for all $\ell\in \N$, $K\in \mu\,\N^{*}\bigcap \N^{*}$, and  $\l\in \C^{*}/q^{K^{-1}\Z}$, we obtain $$\widehat{\mathcal{L}}_{\mu}\left(\z^{\ell}\right)=\mathcal{L}_{\mu,K}^{[\l]}\left(\z^{\ell}\right).$$
\pagebreak[3]
\begin{propo}\label{propo1}
Let us consider $(\mu,K)\in \Q_{>0}\times \N^{*}$ with $K\in \mu\,\N^{*}$, $\l\in \C^{*}/q^{K^{-1}\Z}$, $f\in\mathcal{M}(\C^{*},0)$ and ${\hat{f}:=\displaystyle\sum_{i\in \N}\hat{f}_{i}\z^{i}\in \C[[\z]]}$ such that $f\displaystyle\sim^{[\l]}_{\mu,K}\hat{f}$. Let $(\mu_{1},K_{1})\in \Q_{>0}\times \N^{*}$ with $K_{1}\in \mu_{1}\,\N^{*}$ and $K/K_{1}\in \N^{*}$, let us choose an identification of $\l$, as an element of $\C^{*}/q^{K_{1}^{-1}\Z}$ and assume that $f\in  \mathbb{H}^{[\l]}_{\mu_{1},K_{1}}$. Then, we have
$$\mathcal{L}_{\mu_{1},K_{1}}^{[\l]}\left( f\right)\displaystyle\sim^{[\l]}_{\mu_{2},K_{1}}\widehat{\mathcal{L}}_{\mu_{1}}\left(\hat{f}\right), $$
where 
$$\mu_{2}^{-1}:=\mu^{-1}+\mu_{1}^{-1}.$$
\end{propo}

\begin{proof}
Let us fix $\e,R>0$ sufficiently small. Due to $f\displaystyle\sim^{[\l]}_{\mu,K}\hat{f}$, there exist $L,M>0$, such that for all $k\in \N$ and for all $\z$ in
$$\Big\{\z\in \C^{*}\Big| |\z|<R\Big\}\setminus \displaystyle\bigcup_{\ell\in K^{-1}\Z}\Big\{\z\in \C^{*}\Big| \left|\z+q^{\ell}\l \right|<\e\left| q^{\ell}\l \right|\Big\}:=\G,$$
we have 
\begin{equation}\label{eq8}
\Bigg|f(\z)-\displaystyle\sum_{i=0}^{k-1}\hat{f}_{i}\z^{i} \Bigg|<LM^{k}|q|^{\frac{k(k-1)}{2\mu}}|\z|^{k}.
\end{equation}
Let 
$\G_{1}:=\Big\{z\in \C^{*}\Big| |z|<R\Big\}\setminus \displaystyle\bigcup_{\ell\in K_{1}^{-1}\Z}\Big\{z\in \C^{*}\Big| \left|z+q^{\ell}\l \right|<\e\left| q^{\ell}\l \right| \Big\}$.
We want to prove the existence of $C\in \R_{>0}$ such that for all $z\in  \G_{1}$, for all $k\in \N$, we have 
$$\left|\mathcal{L}_{\mu_{1},K_{1}}^{[\l]}\left(f\right)(z)-\widehat{\mathcal{L}}_{\mu_{1}}\left(\hat{f}^{(k)}\right)(z)\right|\leq CM^{k}|q|^{\frac{k(k-1)}{2\mu_{2}}}|z|^{k}.$$

Let us fix $k\in \N$ and let~$\hat{f}^{(k)}:=\displaystyle\sum_{i=0}^{k-1}\hat{f}_{i}\z^{i}$. In virtue of Remark \ref{rem4}, we find that, for all~$z\in \G_{1}$, we have the equality 
$$ \left|\mathcal{L}_{\mu_{1},K_{1}}^{[\l]}\left(f\right)(z)-\widehat{\mathcal{L}}_{\mu_{1}}\left(\hat{f}^{(k)}\right)(z)\right|=\left|\mathcal{L}_{\mu_{1},K_{1}}^{[\l]}\left(f-\hat{f}^{(k)}\right)(z)\right|.$$  
Since $K/K_{1}\in \N^{*}$, we have $\G_{1}\subset \G$. Then, the inequality (\ref{eq8}) can be used in the following computation:

$$\begin{array}{cl}
&\left|\mathcal{L}_{\mu_{1},K_{1}}^{[\l]}\left(f-\hat{f}^{(k)}\right)(z)\right|\\\\
\leq&\dfrac{\mu_{1}}{K_{1}}
\displaystyle \sum_{\ell\in K_{1}^{-1}\Z}\left|\frac{f(q^{\ell}\l)-\displaystyle\sum_{i=0}^{k-1}\hat{f}_{i}q^{i\ell}\l^{i}}
{\T_{q^{1/\mu_{1}}}\left(\frac{q^{\frac{1}{\mu_{1}}+\ell}\l}{z}\right)}\right|
\\\\

\leq&LM^{k}|q|^{\frac{k(k-1)}{2\mu}}\dfrac{\mu_{1}}{K_{1}}
\displaystyle \sum_{\ell\in K_{1}^{-1}\Z}\left|\frac{q^{\ell k}\l^{k}}
{\T_{q^{1/\mu_{1}}}\left(\frac{q^{\frac{1}{\mu_{1}}+\ell}\l}{z}\right)}\right|.\end{array}$$
By straightforward computations, we find that the latter quantity is equal to

$$\begin{array}{cl}
&LM^{k}|q|^{\frac{k(k-1)}{2\mu}}\dfrac{\mu_{1}}{K_{1}}\displaystyle \sum_{j=1}^{K_{1}/\mu_{1}}
\left|\T_{q^{1/\mu_{1}}}\left(\frac{q^{\frac{j}{K_{1}}}\l}{z}\right)\right|^{-1}\displaystyle \sum_{\ell\in \Z}\left|\dfrac{q^{\frac{jk}{K_{1}}}q^{\frac{\ell k}{\mu_{1}}}\l^{k}z^{\ell}}
{q^{\frac{ k}{\mu_{1}}}q^{\frac{j\ell}{K_{1}}}q^{\frac{\ell(\ell+1)}{2\mu_{1}}}\l^{\ell}}\right|\\\\

\leq&LM^{k}|q|^{\frac{k(k-1)}{2\mu}}\dfrac{\mu_{1}}{K_{1}}\displaystyle \sum_{j=1}^{K_{1}/\mu_{1}}\left|\T_{q^{1/\mu_{1}}}\left(\frac{z}{q^{1+\frac{j}{K_{1}}}\l}\right)\right|^{-1}
\left|\frac{q^{\frac{jk}{K_{1}}}\l^{k}}{q^{\frac{k}{\mu_{1}}}}\T_{|q|^{1/\mu_{1}}}\left(\left|\dfrac{q^{\frac{k}{\mu_{1}}}z}{q^{1+\frac{j}{K_{1}}}\l}\right|\right)\right|\\\\

\leq&LM^{k}|q|^{\frac{k(k-1)}{2\mu}}|q|^{\frac{k(k-1)}{2\mu_{1}}}|z|^{k}\dfrac{\mu_{1}}{K_{1}}\displaystyle \sum_{j=1}^{K_{1}/\mu_{1}}\left|\dfrac{q^{\frac{j}{K_{1}}}\T_{|q|^{1/\mu_{1}}}\left(\left|\frac{z}{q^{1+j/K_{1}}\l}\right|\right)}{q^{\frac{k}{\mu_{1}}}\T_{q^{1/\mu_{1}}}\left(\frac{z}{q^{1+j/K_{1}}\l}\right)}\right|.
\end{array}$$
For $j\in \{1,\dots,K_{1}/\mu_{1}\}$, let us define $f_{j}(z):=\left|\dfrac{q^{\frac{j}{K_{1}}}\T_{|q|^{1/\mu_{1}}}\left(\left|\frac{z}{q^{1+j/K_{1}}\l}\right|\right)}{q^{\frac{k}{\mu_{1}}}\T_{q^{1/\mu_{1}}}\left(\frac{z}{q^{1+j/K_{1}}\l}\right)}\right|$ which is invariant under the action of $\sq$ and is continuous on $\G_{1}$. With the same reasoning as in the proof of Lemma~\ref{lem2}, we establish the existence of $C_{0}>0$, such that for all $z\in \G_{1}$, and for all $j\in \{1,\dots,K_{1}/\mu_{1}\}$, 
$$f_{j}(z)<C_{0}.$$
Hence, 
$$\begin{array}{cl}
&LM^{k}|q|^{\frac{k(k-1)}{2\mu}}|q|^{\frac{k(k-1)}{2\mu_{1}}}|z|^{k}\dfrac{\mu_{1}}{K_{1}}\displaystyle \sum_{j=1}^{K_{1}/\mu_{1}} f_{j}\\\\
\leq&C_{0}LM^{k}|q|^{\frac{k(k-1)}{2\mu}}|q|^{\frac{k(k-1)}{2\mu_{1}}}|z|^{k}=C_{0}LM^{k}|q|^{\frac{k(k-1)}{2\mu_{2}}}|z|^{k}.
\end{array}$$
This completes the proof.
\end{proof}

\pagebreak[3]
\subsection{Main result.}\label{sec12}

The goal of this subsection is to prove that if a linear $q$-difference equation with rational coefficients admits a formal power series $\hat{h}$ as a solution, then we may apply to $\hat{h}$ several $q$-Borel and $q$-Laplace transformations of convenient orders and convenient direction in order to obtain a solution of the same equation that belongs to~$\mathcal{M}(\C^{*})$. 
First, we introduce some notations.\\ \par
 For $R\in \{\C[z],\C(z)\}$ and $n\in \N^{*}$, we define $\mathcal{D}_{R,n}$ as the ring of $q$-difference operators of the form:
$$\displaystyle\displaystyle \sum_{i=l}^{m} b_{i}\sq^{i/n},$$
where $b_{i}\in R$, $l,m\in \Z$, $l\leq m$, and $\sq^{i/n}:=\sigma_{q^{i/n}}$. Let $\mathcal{D}_{R,\infty}:=\displaystyle \bigcup_{n\in \N^{*}}\mathcal{D}_{R,n}$. To simplify the notations, we will write $\mathcal{D}_{R}$ instead of $\mathcal{D}_{R,1}$.\\ \par
Let $P\in \mathcal{D}_{R,\infty}$. 
The Newton polygon of $P$ is the convex hull of the $$\displaystyle\bigcup_{i=l}^{m}\Big\{ (i,j)\in \Q\times\Z \Big| j\geq v_{0}(b_{i})\Big\},$$ where $v_{0}$ denotes the $z$-adic valuation. Let $(d_{1},n_{1}),\dots ,(d_{r+1},n_{r+1})$ with ${d_{1}<\dots<d_{r+1}}$, be a minimal subset of~$\Q\times\Z$ for the inclusion, such that the lower part of the boundary of the Newton polygon of $P$ is the convex hull of $(d_{1},n_{1}),\dots ,(d_{r+1},n_{r+1})$. If $r>0$, we call slopes of $P$ the rational numbers $\frac{n_{i+1}-n_{i}}{d_{i+1}-d_{i}}$, and multiplicity of the slope $\frac{n_{i+1}-n_{i}}{d_{i+1}-d_{i}}$ the integer $d_{i+1}-d_{i}$. We call positive slopes of $P$, the set of slopes that belongs to $\Q_{> 0}$. If $r=0$, we make the convention that the $q$-difference equation has an unique slope equals to $0$.\\\par 

Let $\hat{h}\in \C[[z]]\setminus\{0\}$ be a solution of a linear $q$-difference equation in coefficients in $\C(z)$. Let $\C((z))$ be the fraction field of  $\C[[z]]$ and let $m+1\in \N^{*}$ be the dimension of the $\C(z)$-vectorial subspace of $\C((z))$, spanned by $1$ and $\left\{\sq^{i}\left(\hat{h}\right),i\in \N\right\}$. Let ${P=a_{m}\sq^{m}+a_{m-1}\sq^{m-1}+\dots+a_{0}\in \mathcal{D}_{\C[z]}}$ and $a\in \C[z]$ with $\gcd(a,a_{0},\dots,a_{m})=1$ and ${P\left(\hat{h}\right)=a}$. 
Until the end of the subsection, we assume that $P$ has slopes strictly bigger than~$0$. Let~${\mu_{1}<\dots<\mu_{r}}$ be the positive slopes of the equation and set $\mu_{r+1}:=+\infty$. Let~$(\kappa_{1},\dots,\kappa_{r})$ be defined as:
$$\kappa_{i}^{-1}:=\mu_{i}^{-1}-\mu_{i+1}^{-1}.$$
Let $K\in \N^{*}$ (resp. $n\in \N^{*}$) be minimal, such that for all $i\in \{1,\dots,r\}$, $\frac{K}{\k_{i}}\in \N^{*}$ (resp. such that $\frac{n}{\k_{r}}\in \N^{*}$). Note that $\frac{K}{n}\in \N^{*}$, and therefore  $\mathbb{H}_{\k_{r},K}\subset \mathbb{H}_{\k_{r},n}$. In what follows, we are going to identify the elements of $\C^{*}/q^{n^{-1}\Z}$ as elements of $\C^{*}/q^{K^{-1}\Z}$. We now state the main result of the paper.
\pagebreak[3]
\begin{theo}\label{theo1}
 There exists a finite set $\Sigma\subset \C^{*}/q^{n^{-1}\Z}$, such that if $\l\in \left(\C^{*}/q^{n^{-1}\Z}\right)\setminus\Sigma$, then $$\mathcal{\hat{B}}_{\k_{1}}\circ \dots\circ\mathcal{\hat{B}}_{\k_{r}}\left(\hat{h}\right)\in \mathbb{H}_{\k_{1},K}^{[\l]}$$
and 
for $i=1$ (resp. $i=2,\dots,i=r-1$),
$\mathcal{L}_{\k_{i},K}^{[\l]}\circ\dots\circ \mathcal{L}_{\k_{1},K}^{[\l]}\circ\mathcal{\hat{B}}_{\k_{1}}\circ \dots\circ\mathcal{\hat{B}}_{\k_{r}}\left(\hat{h}\right)\in \mathbb{H}_{\k_{i+1},K}^{[\l]}$. Moreover, the following function
$$S_{q}^{[\l]}\left(\hat{h}\right):=\mathcal{L}_{\k_{r},n}^{[\l]}\circ\mathcal{L}_{\k_{r-1},K}^{[\l]}\circ\dots\circ \mathcal{L}_{\k_{1},K}^{[\l]}\circ\mathcal{\hat{B}}_{\k_{1}}\circ \dots\circ\mathcal{\hat{B}}_{\k_{r}}\left(\hat{h}\right)\in \mathcal{M}(\C^{*}),$$
 is solution of
$$P\left(\hat{h}\right)=P\left(S_{q}^{[\l]}\left(\hat{h}\right)\right)=a\in \C[z].$$
\end{theo}
\pagebreak[3]
\begin{rem}\label{rem2}
For $|z|$ close to~$0$, $S_{q}^{[\l]}\left(\hat{h}\right)$ has poles of order at most $1$ that are contained in the $q^{1/n}$-spiral $-q^{n^{-1}\Z}\l$.\end{rem}

A similar result was proved by Marotte and Zhang in \cite{MZ}. They have used the analytic factorization of $P$ to deduce a summation theorem. We are going to prove the theorem by a purely algebraical argument based on (\ref{eq2}). To illustrate the strategy of the proof, we treat explicitly the following example.
\pagebreak[3]
\begin{ex}
Let $\hat{h}\in \C[[z]]$ be a solution of $P\left(\hat{h}\right)=1$ with
$$P:=z^{4}\sq^{4}+z\sq^{2}+\sq,$$
and assume that the dimension of the $\C(z)$-vectorial subspace of $\C((z))$, spanned by $1$ and $\left\{\sq^{i}\left(\hat{h}\right),i\in \N\right\}$ is $4$.
The slopes are~$1$ and $3/2$. 
With (\ref{eq2}), we find $Q\left(\hat{\mathcal{B}}_{3/2}\left(\hat{h}\right)\right)=1$, with 
$$Q:=(q^{-4}\z^{4}+\z)\sq^{\frac{4}{3}}+\sq,$$
which has slope $3$. We again use (\ref{eq2}) to find 
$R\left(\hat{\mathcal{B}}_{3}\circ\hat{\mathcal{B}}_{3/2}\left(\hat{h}\right)\right)=1$, with 
$$R:=(\z+1)\sq+q^{-6}\z^{4}\sq^{0},$$
which has slope $-4$. 
 In particular, $\hat{\mathcal{B}}_{3}\circ\hat{\mathcal{B}}_{3/2}(\hat{h})\in \C\{\z\}$. The $q$-difference equation satisfied by this latter function implies that if $\l\in \left(\C^{*}/q^{3^{-1}\Z}\right)\setminus\{ -1\}$, then $\hat{\mathcal{B}}_{3}\circ\hat{\mathcal{B}}_{3/2}(\hat{h})\in\mathbb{H}_{3,3}^{[\l]}$. With Remark~\ref{rem1}, we obtain that for such a $\l$ :
$$Q\left(\mathcal{L}_{3,3}^{[\l]}\circ\hat{\mathcal{B}}_{3}\circ\hat{\mathcal{B}}_{3/2}(\hat{h})\right)=1. $$
Using the same reasoning, we deduce that if $$\l\in \left(\C^{*}/q^{3^{-1}\Z}\right)\setminus\left\{ -1,e^{\frac{i\pi}{3}},e^{\frac{-i\pi}{3}}\right\},$$ then $\mathcal{L}_{3,3}^{[\l]}\circ\hat{\mathcal{B}}_{3}\circ\hat{\mathcal{B}}_{3/2} \left(\hat{h}\right)\in\mathbb{H}_{3/2,3}^{[\l]}$, and 
$$P\left(S_{q}^{[\l]}\left(\hat{h}\right)\right)=1,$$
where $$S_{q}^{[\l]}\left(\hat{h}\right):=\mathcal{L}_{3/2,3}^{[\l]}\circ\mathcal{L}_{3,3}^{[\l]}\circ\hat{\mathcal{B}}_{3}\circ\hat{\mathcal{B}}_{3/2}\left(\hat{h}\right).$$
$$\begin{tabular}{ccc}
\includegraphics[width=0.287\linewidth]{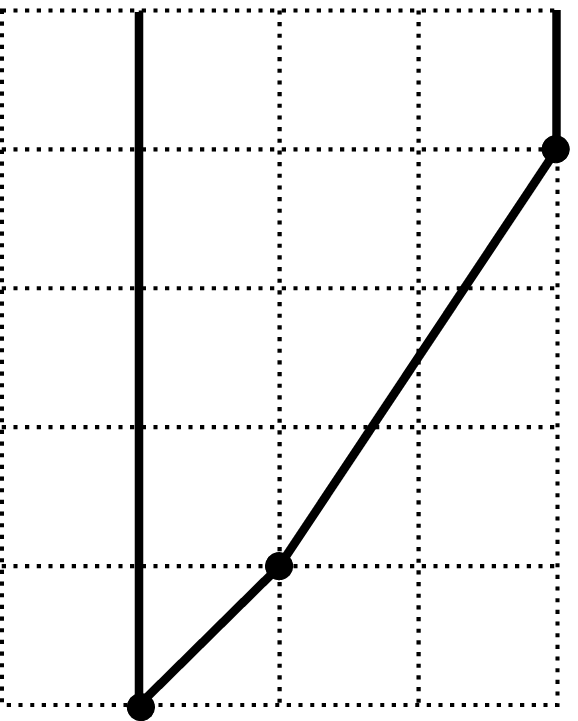}
&
\includegraphics[width=0.283\linewidth]{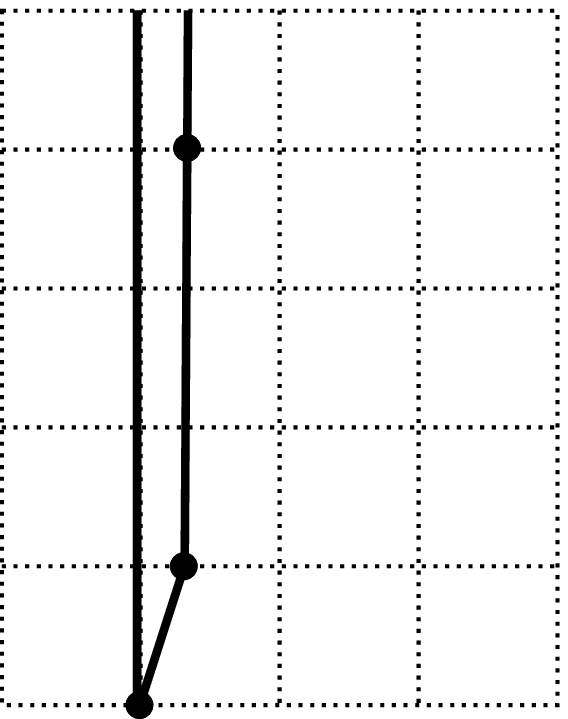}
&
\includegraphics[width=0.283\linewidth]{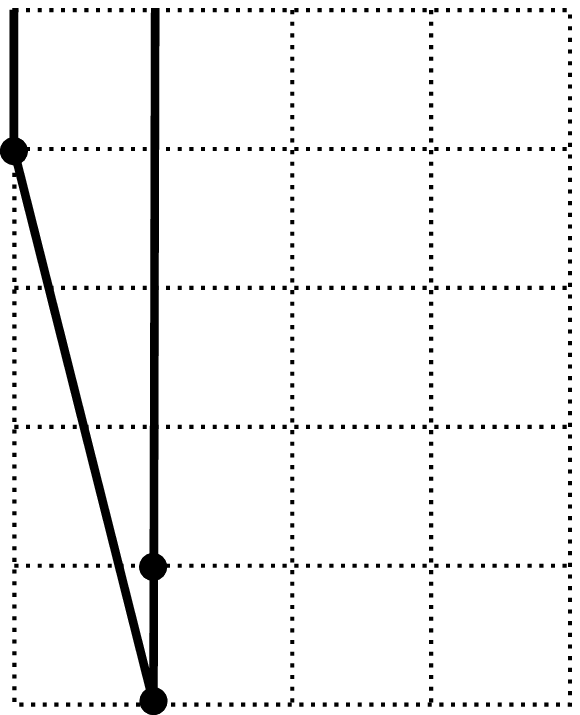}
 \\
\\
Newton polygon of  $P$& Newton polygon of $Q$&Newton polygon of $R$
\end{tabular}$$

\end{ex}

Before proving the theorem, we need to state and prove a proposition.

\pagebreak[3]
\begin{propo}\label{propo2}
Let $K\in \N^{*}$, $\hat{h}\in\C[[z]]$, and $P\in\mathcal{D}_{\C[z],K}$ with maximal slopes strictly bigger than~$0$, such that $P\left(\hat{h}\right)=a\in \C[z]$. Let ${\mu_{1}<\dots<\mu_{r}}$, be the positive slopes of $P$, let $\mu\geq \mu_{r}$ be a rational number, and let us assume that $\frac{K}{\mu}\in \N^{*}$. The following statement hold. 
\begin{trivlist}
\item (1)
There exists $Q:=\displaystyle\sum_{\substack{i=m_{1}\\i\in K^{-1}\Z}}^{m_{2}} b_{i}(\z)\sq^{i}\in\mathcal{D}_{\C[\z],K}$, such that the series $\hat{\mathcal{B}}_{\mu}\left(\hat{h}\right)$ satisfies $Q\left(\hat{\mathcal{B}}_{\mu}\left(\hat{h}\right)\right)=\hat{\mathcal{B}}_{\mu}\left(a\right)$. Moreover:
\begin{itemize}
\item The positive slopes of $Q$ are equal to $\left(\mu_{1}^{-1}-\mu^{-1}\right)^{-1},\dots,\left(\mu_{r}^{-1}-\mu^{-1}\right)^{-1},$ if  $\mu> \mu_{r},$
\item The positive slopes of $Q$ are equal to $\left(\mu_{1}^{-1}-\mu^{-1}\right)^{-1},\dots,\left(\mu_{r-1}^{-1}-\mu^{-1}\right)^{-1}$, if   ${\mu= \mu_{r}, r>1,}$
\item $Q$ has maximal slope equal to $0$, if $\mu= \mu_{r}, r=1$,
\end{itemize}

\item (2) For all $i\in [m_{1},m_{2}]\cap K^{-1}\Z$, we have \begin{equation}\label{eq4}
\deg \left(\frac{b_{i}}{b_{m_{2}}}\right)\leq (m_{2}-i)\mu.
\end{equation}
\item (3)
 Any solution of $Q\left(\hat{\mathcal{B}}_{\mu}\left(\hat{h}\right)\right)=\hat{\mathcal{B}}_{\mu}\left(a\right)$ that belongs to $\mathcal{M}(\C^{*},0)$ belongs also to $\mathbb{H}_{\mu,K}$.
\end{trivlist}
\end{propo}

\begin{proof}
\begin{trivlist}
\item (1)
 We compute explicitly a linear $q$-difference equation satisfied by $\hat{\mathcal{B}}_{\mu}\left(\hat{h}\right)$ using Lemma \ref{lem1}. Let~$(s_{1},t_{1}),\dots ,(s_{k},t_{k})$  with ${s_{1}<\dots<s_{k}}$, be a minimal subset of~$\Q\times \Z$ for the inclusion, such that~$(s_{1},t_{1}),\dots ,(s_{k},t_{k})$ is the lower part of the boundary of the Newton polygon of~$P$. Since the Newton polygon has $r$ positive slopes, it follows that ${t_{k-r}<\dots<t_{k}}$. Let us write $$P=:\displaystyle \sum_{\substack{i=s_{1}\\i\in K^{-1}\Z}}^{s_{k}}\sum_{j=0}^{k_{i}} a_{i,j}z^{j}\sq^{i},$$
with $a_{i,j}\in \C$ and $k_{i}\in \N$. 
Using  (\ref{eq2}), we find the existence of $Q\in\mathcal{D}_{\C[\z],K}$, such that ${Q\left(\hat{\mathcal{B}}_{\mu}\left(\hat{h}\right)\right)=\hat{\mathcal{B}}_{\mu}\left(a\right)}$, where
\begin{equation}\label{eq1}
Q:=\displaystyle \sum_{\substack{i=s_{1}\\i\in K^{-1}\Z}}^{s_{k}}\sum_{j=0}^{k_{i}} \dfrac{a_{i,j}\z^{j}\sq^{i-j/\mu}}{q^{\frac{j(j-1)}{2\mu}}}.\end{equation}
Consequently,  
\begin{itemize}
\item $Q$ has positive slopes equal to $$\dfrac{t_{k+1-r}-t_{k-r}}{s_{k+1-r}-s_{k-r}-\frac{t_{k+1-r}-t_{k-r}}{\mu}},\dots,\dfrac{t_{k}-t_{k-1}}{s_{k}-s_{k-1}-\frac{t_{k}-t_{k-1}}{\mu}}, \hbox{ if }\mu>\mu_{r},$$
\item $Q$  has positive slopes equal to
$$\dfrac{t_{k+1-r}-t_{k-r}}{s_{k+1-r}-s_{k-r}-\frac{t_{k+1-r}-t_{k-r}}{\mu}},\dots,\dfrac{t_{k-1}-t_{k-2}}{s_{k-1}-s_{k-2}-\frac{t_{k-1}-t_{k-2}}{\mu}} \hbox{ if } \mu= \mu_{r}, r>1,$$
\item $Q$ has maximal slope equal to $0$, if $\mu=\mu_{r}, r=1$.
\end{itemize}
Those latter equal to 
$$\left(\mu_{1}^{-1}-\mu^{-1}\right)^{-1},\dots,\left(\mu_{r}^{-1}-\mu^{-1}\right)^{-1}, \hbox{ if } \mu>\mu_{r}$$ 
resp. 
$$\left(\mu_{1}^{-1}-\mu^{-1}\right)^{-1},\dots,\left(\mu_{r-1}^{-1}-\mu^{-1}\right)^{-1}, \hbox{ if } \mu=\mu_{r},r>1.$$
\item (2)
Let us write $Q:=\displaystyle\sum_{\substack{i=m_{1}\\i\in K^{-1}\Z}}^{m_{2}} b_{i}(\z)\sq^{i}$. Following the proof of (1), we find $m_{2}:=s_{k}-\frac{t_{k}}{\mu}$. Using (\ref{eq1}), we find that the $\z$-degree of~$b_{m_{2}}$ is $t_{k}$, and the $\z$-degree of the other $b_{i}$ are bounded by $t_{k}+(m_{2}-i)\mu$. Hence, for all $i\in [m_{1},m_{2}]\cap K^{-1}\Z$, 
$$
\deg \left(\frac{b_{i}}{b_{m_{2}}}\right)\leq (m_{2}-i)\mu.
$$
This proves (2).
\item (3)
Assume now the existence of $f \in \mathcal{M}(\C^{*},0)$ solution of $Q\left(f\right)=0$. Let us prove the existence of a finite set $\Sigma\subset \C^{*}/q^{K^{-1}\Z}$, such that for all $\l\in \left(\C^{*}/q^{K^{-1}\Z}\right)\setminus \Sigma$, there exist $M_{0}\in \C^{*}$, $L,\e>0$, such that $f$ admits an analytic continuation on 
$$\displaystyle\bigcup_{\ell\in K^{-1}\Z}\Big\{\z\in \C^{*}\Big| \left|\z-\l q^{\ell}\right|<\e\left| q^{\ell}\l \right|\Big\},$$ 
that satisfies 
\begin{equation}\label{eq3}
|f(\z)|<L\left|\T_{|q|^{1/K}}(M_{0}|\z|)\right|^{\mu/K}.
\end{equation}
Let $M\in \C^{*}$. The function, $f(\z)\T_{|q|^{1/K}}(M|\z|)^{-\mu/K}$ satisfies a linear $q$-difference equation 
\begin{equation}\label{eq6}
R_{M}\left(f(\z)\T_{|q|^{1/K}}(M|\z|)^{-\mu/K}\right)=\frac{\hat{\mathcal{B}}_{\mu}\left(a\right)}{b_{m_{2}}(\z)M^{m_{2}\mu}\left|\z\right|^{m_{2}\mu}\T_{|q|^{1/K}}(M|\z|)^{\mu/K}}
\end{equation}
with 
$$R_{M}:=\displaystyle \sum_{\substack{i=m_{1}\\i\in K^{-1}\Z}}^{m_{2}} \frac{b_{i}(\z)M^{i\mu}\left|\z\right|^{i\mu}}
{b_{m_{2}}(\z)M^{m_{2}\mu}\left|\z\right|^{m_{2}\mu}}\sq^{i}=:
\displaystyle \sum_{\substack{i=m_{1}\\i\in K^{-1}\Z}}^{m_{2}} c_{i,M}(\z)\sq^{i}.$$
Due to (\ref{eq4}), for all $M\in \C^{*}$, the $c_{i,M}(\z)$ are bounded for $|\z|$ big. Note that  for all $M\in \C^{*}$, for all $i\in K^{-1}\Z$, we have:
$$c_{i,M}=M^{(i-m_{2})\mu}c_{i,1}.$$
Consequently, there exists $M_{0}\in \C^{*}$ such that $\left|\displaystyle\sum_{\substack{i=m_{1}\\i\in K^{-1}\Z}}^{m_{2}-1/K} c_{i,M_{0}}(\z)\right|$ is bounded by $\frac{1}{2K(m_{2}-m_{1})}$ for $|\z|$ big. The $q$-difference equation satisfied by the right-hand side of (\ref{eq6}) implies that it tends to~$0$ as $|\z|$ tends to infinity.
We recall that by construction, $c_{m_{2},M_{0}}=1$. From~(\ref{eq6}) and the previous facts, we obtain that for all $\z\in \C^{*}$ with $\z q^{K^{-1}\Z}$ does not intersect the poles of the $c_{i,M_{0}}$, there exists a constant $L>0$, such that we have for all $\ell\in \Z$:
$$\left|f\left(\z q^{\ell/K}\right)\T_{|q|^{1/K}}\left(M_{0}\left|\z q^{\ell/K}\right|\right)^{-\mu/K}\right|<L.$$
In particular, this yields (\ref{eq3}).
The $q$-difference equations satisfied by $\T_{|q|^{1/K}}(M_{0}|\z|)^{\mu/K}$ and $\T_{|q|^{1/\mu}}(|M_{0}\z|)$ imply that the function $\left|\frac{\T_{|q|^{1/K}}\big(M_{0}|\z|\big)^{\mu/K}}{\T_{|q|^{1/\mu}}\big(|M_{0}\z|\big)}\right|$ is bounded. Hence
$$f\in\mathbb{H}_{\mu,K}.$$
\end{trivlist}
\end{proof}

\begin{proof}[Proof of Theorem \ref{theo1}]
Applying successively Proposition \ref{propo2} $r$ times,
$\mathcal{\hat{B}}_{\k_{1}}\circ \dots\circ\mathcal{\hat{B}}_{\k_{r}}\left(\hat{h}\right)$ has maximal slope equals to $0$ and therefore it converges. Proposition \ref{propo2} yields that it belongs also to $\mathbb{H}_{\k_{1},K}$. Let $\l\in \C^{*}/q^{K^{-1}\Z}$
such that the convergent series belongs to $\mathbb{H}^{[\l]}_{\k_{1},K}$.
Because of Remark~\ref{rem1}, for $i=1$, the function
$$\mathcal{L}_{\k_{i},K}^{[\l]}\circ\dots\circ \mathcal{L}_{\k_{1},K}^{[\l]}\circ\mathcal{\hat{B}}_{\k_{1}}\circ \dots\circ\mathcal{\hat{B}}_{\k_{r}}\left(\hat{h}\right)\in \mathcal{M}(\C^{*},0)$$  satisfies the same linear $q$-difference equation as the formal power series $\mathcal{\hat{B}}_{\k_{i+1}}\circ \dots\circ\mathcal{\hat{B}}_{\k_{r}}\left(\hat{h}\right)$. Using additionally Proposition \ref{propo2}, we obtain that for $i=1$, 
$$\mathcal{L}_{\k_{i},K}^{[\l]}\circ\dots\circ \mathcal{L}_{\k_{1},K}^{[\l]}\circ\mathcal{\hat{B}}_{\k_{1}}\circ \dots\circ\mathcal{\hat{B}}_{\k_{r}}\left(\hat{h}\right)\in \mathbb{H}_{\k_{i+1},K}.$$ 
We apply successively the same reasoning for ${i=2,\dots,i=r-1}$. We obtain the existence of $\Sigma'$, a finite subset of $\C^{*}/q^{K^{-1}\Z}$, such that for all ${\l\in \left(\C^{*}/q^{K^{-1}\Z}\right)\setminus \Sigma'}$, the following composition of functions makes sense
$$\mathcal{L}_{\k_{r-1},K}^{[\l]}\circ\dots\circ \mathcal{L}_{\k_{1},K}^{[\l]}\circ\mathcal{\hat{B}}_{\k_{1}}\circ \dots\circ\mathcal{\hat{B}}_{\k_{r}}\left(\hat{h}\right)\in  \mathbb{H}_{\k_{r},K}.$$
We recall that $\frac{K}{n}\in \N^{*}$, which implies in particular that $\mathbb{H}_{\k_{r},K}\subset \mathbb{H}_{\k_{r},n}$. Therefore, there exists $\Sigma$, a finite subset of $\C^{*}/q^{n^{-1}\Z}$, such that for all ${\l\in \left(\C^{*}/q^{n^{-1}\Z}\right)\setminus \Sigma}$, the following composition of functions makes sense
$$S_{q}^{[\l]}\left(\hat{h}\right):=\mathcal{L}_{\k_{r},n}^{[\l]}\circ\mathcal{L}_{\k_{r-1},K}^{[\l]}\circ\dots\circ \mathcal{L}_{\k_{1},K}^{[\l]}\circ\mathcal{\hat{B}}_{\k_{1}}\circ \dots\circ\mathcal{\hat{B}}_{\k_{r}}\left(\hat{h}\right)\in \mathcal{M}(\C^{*},0).$$
Using again Remark~\ref{rem1}, we deduce that $S_{q}^{[\l]}\left(\hat{h}\right)$ satisfies the same linear $q$-difference equation as~$\hat{h}$. Since $S_{q}^{[\l]}\left(\hat{h}\right)\in \mathcal{M}(\C^{*},0)$ is solution of a linear $q$-difference equation with rational coefficients, we obtain that $S_{q}^{[\l]}\left(\hat{h}\right)$ belongs to $ \mathcal{M}(\C^{*})$. This concludes the proof.
\end{proof}

Using Proposition \ref{propo1}, we find that the meromorphic solution of Theorem \ref{theo1} is asymptotic to $\hat{h}$.
\pagebreak[3]
\begin{propo}\label{propo3}
Let us keep the same notations as in Theorem \ref{theo1}. For all ${\l\in \left(\C^{*}/q^{n^{-1}\Z}\right)\setminus \Sigma}$, we have
$$S_{q}^{[\l]}\left(\hat{h}\right)\displaystyle\sim^{[\l]}_{\mu_{r},n}\hat{h}.$$
\end{propo}

\pagebreak[3]
\begin{rem}
Let us keep the same notations and let ${\l\in \left(\C^{*}/q^{n^{-1}\Z}\right)\setminus \Sigma}$. We wonder if the map $\hat{h}\mapsto S_{q}^{[\l]}\left(\hat{h}\right)$ has algebraic properties. We give here a partial answer. Let us consider $f\in \C(z)$ such that $f\hat{h}\in \C[[z]]$. Since the slopes of the linear $q$-difference equations satisfied by $\hat{h}$ and $f\hat{h}$ are the same, we obtain using Remark~\ref{rem1}, that there exists a finite set $\Sigma'\subset \C^{*}/q^{n^{-1}\Z}$, such that for all $\l\in \left(\C^{*}/q^{n^{-1}\Z}\right)\setminus\Sigma'$,  
$$S_{q}^{[\l]}\left(f\hat{h}\right)=fS_{q}^{[\l]}\left(\hat{h}\right).$$
\end{rem}

Let us now state a similar result for the system, which will be used in $\S \ref{sec2}$. Let us consider the vector ${\hat{Y}=\left(\hat{Y}_{j}\right)_{j\leq m}\in \Big(\C[[z]]\Big)^{m}}$ of formal power series, solution of ${\sq \hat{Y}= A\hat{Y}}$, with~${A\in \mathrm{GL}_{m}(\C(z))}$. Let~${\{\mu_{1},\dots,\mu_{r}\}\in (\Q_{>0})^{r}}$, be minimal in $r$ for the inclusion, such that each entries of the vector $\hat{Y}$ satisfy a linear $q$-difference equation in coefficients in~$\C[z]$ with positive slopes contained in $\{\mu_{1},\dots,\mu_{r}\}$. Without loss of generality, we may assume that $\mu_{1}<\dots<\mu_{r}$. Set $\mu_{r+1}:=+\infty$. Let~$(\kappa_{1},\dots,\kappa_{r})$ be defined as:
$$\kappa_{i}^{-1}:=\mu_{i}^{-1}-\mu_{i+1}^{-1}.$$
Let $K\in \N^{*}$ (resp. $n\in \N^{*}$) be minimal , such that for all $i\in \{1,\dots,r\}$, $\frac{K}{\k_{i}}\in \N^{*}$ (resp. such that $\frac{n}{\k_{r}}\in \N^{*}$). 

\pagebreak[3]
\begin{theo}\label{theo4} There exists a finite set $\Sigma\subset \C^{*}/q^{n^{-1}\Z}$, such that if $\l\in \left(\C^{*}/q^{n^{-1}\Z}\right)\setminus\Sigma$, then $$\left(\mathcal{\hat{B}}_{\k_{1}}\circ \dots\circ\mathcal{\hat{B}}_{\k_{r}}\left(\hat{Y}_{j}\right)\right)_{j\leq m}\in \left(\mathbb{H}_{\k_{1},K}^{[\l]}\right)^{m}$$

and 
for $i=1$ (resp. $i=2,\dots,i=r-1$),
$$\left(\mathcal{L}_{\k_{i},K}^{[\l]}\circ\dots\circ \mathcal{L}_{\k_{1},K}^{[\l]}\circ\mathcal{\hat{B}}_{\k_{1}}\circ \dots\circ\mathcal{\hat{B}}_{\k_{r}}\left(\hat{Y}_{j}\right)\right)_{j\leq m}\in \left(\mathbb{H}_{\k_{i+1},K}^{[\l]}\right)^{m}.$$ Moreover, the following vector
$$S_{q}^{[\l]}\left(\hat{Y}\right):=\left(\mathcal{L}_{\k_{r},n}^{[\l]}\circ\mathcal{L}_{\k_{r-1},K}^{[\l]}\circ\dots\circ \mathcal{L}_{\k_{1},K}^{[\l]}\circ\mathcal{\hat{B}}_{\k_{1}}\circ \dots\circ\mathcal{\hat{B}}_{\k_{r}}\left(\hat{Y}_{j}\right)\right)_{j\leq m}\in \Big(\mathcal{M}(\C^{*})\Big)^{m},$$
is solution of $\sq S_{q}^{[\l]}\left(\hat{Y}\right)=AS_{q}^{[\l]}\left(\hat{Y}\right),$
and satisfies
$S_{q}^{[\l]}\left(\hat{Y}\right)\displaystyle\sim^{[\l]}_{\mu_{r},n}\hat{Y}$.
\end{theo}

\begin{proof}
The proof is similar to the proof of Theorem \ref{theo1}. The only major difference is that we may need additionally Remark \ref{rem4} to treat the coordinates $\hat{h}$ of $\hat{Y}$ that satisfy $\mathcal{\hat{B}}_{\k_{i}}\circ \dots\circ\mathcal{\hat{B}}_{\k_{r}}\left(\hat{h}\right)\in \C\{\z\}$, for $i\in \{ 2,\dots,r\}$.
\end{proof}

\pagebreak[3]
\section{Applications}\label{sec2}

We give two applications of Theorem \ref{theo4}. In $\S \ref{sec21}$, we obtain an explicit version of a theorem proved by Praagman: we show how we can compute a fundamental solution of a linear $q$-difference system in coefficients in $\C(z)$, of which entries are meromorphic on $\C^{*}$.
In $\S \ref{sec22}$, we explain how the solutions of Theorem \ref{theo4} are related to the solutions present in \cite{RSZ} in a particular case. \\\par

If the entries of the matrix ${\hat{H}=\left(\hat{H}_{i,j}\right)\in \mathrm{GL}_{m}\Big(\C[[z]]\Big)}$ are solutions of linear $q$-difference equations in coefficients in $\C(z)$, then for a convenient $n\in \N^{*}$ and convention choice of $\l\in \C^{*}/q^{n^{-1}\Z}$, we may apply the $q$-Borel and the $q$-Laplace transformations to every entries of $\hat{H}$ as in Theorem~\ref{theo4} and for such a $\l\in \C^{*}/q^{n^{-1}\Z}$, we write
$$S_{q}^{[\l]}\left(\hat{H}\right):=\left(S_{q}^{[\l]}\left(\hat{H}_{i,j}\right)\right) .$$
 
 \pagebreak[3]
\subsection{Computing a meromorphic fundamental solution of a $q$-difference system.}\label{sec21}
 The goal of this subsection is to show how to compute an invertible matrix solution of a linear $q$-difference system in coefficients in $\C(z)$, of which entries are meromorphic on $\C^{*}$.\\\par 
First, we introduce some notations. Note that the Theta function was already introduced but we recall the expression for the reader's convenience.
Let us consider the meromorphic functions on $\C^{*}$, $\T_{q}(z):=\displaystyle \sum_{\ell \in \Z} q^{\frac{-\ell(\ell+1)}{2}}z^{\ell}$, ${\ell_{q}(z):=\frac{z}{\T_{q}(z)}\frac{d}{dz}\T_{q}(z)}$ and  ${\L_{c}(z):=\frac{\T_{q}(z)}{\T_{q}(z/c)}}$, with $c\in \C^{*}$,  
 that satisfy the linear $q$-difference equations:\\
\begin{itemize}
\item $\sq \T_{q}=z\,\T_{q}$.\\
\item  $\sq \,\ell_{q}=\ell_{q} +1$.\\
\item $\sq \,\L_{c}=c\,\L_{c}.$\\
\end{itemize}
Let $C\in \mathrm{GL}_{m}(\C)$ and consider now the decomposition in Jordan normal form ${C=P(DN)P^{-1}}$ where $DN=ND$, $D=\mathrm{Diag}(d_{i})$ is diagonal, $N$ is a nilpotent upper triangular matrix and $P$ is an invertible matrix with complex coefficients. We construct the matrix
$$
\L_{C}:=P\left(\mathrm{Diag}(\L_{d_{i}})e^{\log(N)\ell_{q}}\right)P^{-1}\in \mathrm{GL}_{m}\Big(\C\Big( \ell_{q},\left(\L_{c}\right)_{c \in \C^{*}}\Big)\Big)$$
that satisfies
$$
\sq \L_{C}=C\L_{C}=\L_{C}C.$$
Remark that if $c\in \C^{*}$ and $(c)\in \mathrm{GL}_{1}(\C)$ is the corresponding matrix, then by construction, we have $\L_{(c)}=\L_{c}$.\\\par 
Let $n\in \Z,d\in \N^{*}$, with $\gcd(n,d)=1$, if $n\neq 0$ and $a\in \C^{*}$. We define the $d$ times $d$ diagonal matrix as follows:
$$E_{n,d,a}:=\mathrm{Diag}\Big(\T_{q^{d}}(az)^{n},\dots,\T_{q^{d}}(q^{d-1}az)^{n}\Big).$$
Let $\C((z))$ (resp. $\C(\{z\})$) be the fraction field of  $\C[[z]]$ (resp. $\C\{z\}$) and let ${A,B\in \mathrm{GL}_{m}\Big(\C((z))\Big)}$. The two $q$-difference systems, $\sq Y=AY$ and $\sq Y=BY$ are said to be formally (resp. analytically) equivalent, if there exists $P\in \mathrm{GL}_{m}\Big(\C((z))\Big)$ (resp. $P\in \mathrm{GL}_{m}\Big(\C(\{z\})\Big)$), called the gauge transformation, such that $$B=P[A]_{\displaystyle\sq}:=(\sq P)AP^{-1}.$$ In particular, $$\sq Y=AY\Longleftrightarrow\sq \left(PY\right)=BPY.$$ 
Conversely, if there exist $A,B,P\in \mathrm{GL}_{m}\Big(\C((z))\Big)$
such that
$\sq Y=AY$, $ \sq Z=BZ$ and $Z=PY$, then 
$$B=P\left[ A\right]_{\displaystyle\sq}.$$\par 
 For the proof of the following theorem, see \cite{BuT}, Theorem 1.18. See also \cite{VdPR}, Corollary 1.6.
\pagebreak[3]
\begin{theo}\label{theo2}
Let $B$ that belongs to $\mathrm{GL}_{m}\Big(\C(\{z\})\Big)$. There exist integers ${n_{1},\dots,n_{k}\in \Z}$ we will assume to be in increasing order, ${d_{1},\dots,d_{k},m_{1},\dots,m_{k}\in \N^{*}}$, with~$\gcd(n_{i},d_{i})=1$, ${\sum m_{i}d_{i}=m}$, and 
\begin{itemize}
\item $\hat{H}\in\mathrm{GL}_{m}\Big(\C[[z]]\Big)$,
\item $C_{i}\in \mathrm{GL}_{m_{i}}(\C)$,
\item $a_{i}\in \C^{*}$,
\end{itemize}
 such that $B=\hat{H}[A]_{\displaystyle\sq}$, where $A\in \mathrm{GL}_{m}\Big(\C(z)\Big)$ is defined by
$$\sq \Big(\mathrm{Diag}\Big(\L_{C_{i}}\otimes E_{n_{i},d_{i},a_{i}} \Big)\Big)=A\mathrm{Diag}\Big(\L_{C_{i}}\otimes E_{n_{i},d_{i},a_{i}} \Big),$$
and $\mathrm{Diag}\Big(\L_{C_{i}}\otimes E_{n_{i},d_{i},a_{i}}\Big):=\begin{pmatrix}
\L_{C_{1}}\otimes E_{n_{1},d_{1},a_{1}} &&\\
&\ddots&\\
&&\L_{C_{k}}\otimes E_{n_{k},d_{k},a_{k}} 
\end{pmatrix}$.
\end{theo}

Assume now that $B\in \mathrm{GL}_{m}\big(\C(z)\big)$. We want to construct a fundamental solution of the system~$\sq Y=BY$, of which entries are meromorphic on $\C^{*}$. First, remark that the entries of $\hat{H}$ satisfy linear $q$-difference equations in coefficients in~$\C(z)$ since
$$B=\sq \left(\hat{H}\right)A\hat{H}^{-1}.$$
\pagebreak[3]
\begin{lem}
There exist $n\in \N^{*}$ and a finite set $\widetilde{\Sigma}\subset \C^{*}/q^{n^{-1}\Z}$, such that for all ${\l\in \left(\C^{*}/q^{n^{-1}\Z}\right)\setminus \widetilde{\Sigma}}$, the invertible matrix ${S_{q}^{[\l]}\left(\hat{H}\right)\in \mathrm{GL}_{m}\big(\mathcal{M}(\C^{*})\big)}$, is solution of:
$$B=\sq \left(S_{q}^{[\l]}\left(\hat{H}\right)\right)AS_{q}^{[\l]}\left(\hat{H}\right)^{-1}.$$
\end{lem}

\begin{proof}
Let $n\in \N^{*}$ and $\widetilde{\Sigma}\subset \C^{*}/q^{n^{-1}\Z}$ be such that for all ${\l\in \left(\C^{*}/q^{n^{-1}\Z}\right)\setminus \widetilde{\Sigma}}$, we may apply the $q$ Borel-Laplace summation to every entries of $\hat{H}$ and to $\det \left(\hat{H}\right)$.
  We apply Theorem~\ref{theo4}. We only have to prove that ${S_{q}^{[\l]}\left(\hat{H}\right)}$ is invertible. There exists $\mu\in \Q_{>0}$, such that 
$$\det \left(S_{q}^{[\l]}\left(\hat{H}\right)\right)\sim^{[\l]}_{\mu,n}\det \left(\hat{H}\right)\neq 0.$$
This implies  \[\det \left(S_{q}^{[\l]}\left(\hat{H}\right)\right)\neq 0.\]

\end{proof}

\pagebreak[3]
\begin{coro}\label{coro1}
Let $\l\in \left(\C^{*}/q^{n^{-1}\Z}\right)\setminus \widetilde{\Sigma}$. Then, $$S_{q}^{[\l]}\left(\hat{H}\right)\mathrm{Diag}\Big(\L_{C_{i}}\otimes E_{n_{i},d_{i},a_{i}} \Big)\in \mathrm{GL}_{m}\big(\mathcal{M}(\C^{*})\big)$$ is solution of~$\sq Y=BY$.
\end{coro}

\pagebreak[3]
\subsection{Local analytic classification of linear $q$-difference equations.}\label{sec22}

Consider a linear $q$-difference system in coefficients in $\C(\{z\})$ that will satisfy an additional condition we explicit above.  The next theorem says that after an analytic gauge transformation, we may put this system in the Birkhoff-Guenther normal form. 
\pagebreak[3]
\begin{theo}[\cite{RSZ}, $\S 3.3.2$]
Let $B\in \mathrm{GL}_{m}\Big(\C(\{z\})\Big)$ and let us consider integers $n_{1},\dots,n_{k},d_{1},\dots,d_{k},m_{1},\dots,m_{k}$, and matrices $C_{1},\dots,C_{k}$ given by Theorem \ref{theo2}. Let us assume that all the $d_{i}\in \N^{*}$  are equal to $1$. Then, there exist 
\begin{itemize}
\item $U_{i,j}$, $m_{i}$ times $m_{j}$ matrices with coefficients in $\displaystyle \sum_{\nu=n_{i}}^{n_{j}-1}\C z^{\nu}$,
\item $F\in\mathrm{GL}_{m}\Big(\C\{z\}\Big)$, 
\end{itemize}
such that $B=F[D]_{\displaystyle\sq}$, where:
$$D:=\begin{pmatrix}
z^{n_{1}}C_{1}&\dots&\dots&\dots&\dots\\
0&\ddots&\dots&U_{i,j}&\dots \\
\vdots&\ddots&\ddots&\dots&\dots \\
\vdots&\dots&\ddots&\ddots&\dots \\
0&\dots&\dots&0&z^{n_{k}}C_{k}
\end{pmatrix}. $$
\end{theo}
In $\S 3.3.3$ of \cite{RSZ}, it is shown the existence and the uniqueness of $$\hat{H}:=\begin{pmatrix}
\mathrm {Id}&\dots&\dots&\dots&\dots\\
0&\ddots&\dots&\hat{H}_{i,j}&\dots \\
\vdots&\ddots&\ddots&\dots&\dots \\ 
\vdots&\dots&\ddots&\ddots&\dots \\
0&\dots&\dots&0&\mathrm {Id}
\end{pmatrix}\in \mathrm{GL}_{m}\Big(\C[[z]]\Big),$$ formal gauge transformation, that satisfies
$$D=\hat{H}\Big[\mathrm{Diag} (z^{n_{i}}C_{i})\Big]_{\displaystyle\sq}.$$ 
In \cite{RSZ}, Theorem  6.1.2, it is proved that if for all ${i<j}$ we choose $\l_{i,j}\in \C^{*}/q^{(n_{j}-n_{i})^{-1}\Z}$ such that $$\l_{i,j}^{n_{j}-n_{i}}q^{-\frac{n_{j}-n_{i}-1}{2}}q^{\Z}\neq\frac{\a_{i'}}{\a_{j'}}q^{\Z},$$
for any $\a_{i'}\in \mathrm{Sp}(C_{i})$, that is the spectrum of $C_{i}$, and any $\a_{j'}\in \mathrm{Sp}(C_{j})$, then there exists an unique matrix $$\hat{H}^{\left[(\l_{i,j})\right]}:=\begin{pmatrix}
\mathrm {Id}&\dots&\dots&\dots&\dots\\
0&\ddots&\dots&\hat{H}^{\left[\l_{i,j}\right]}_{i,j}&\dots \\
\vdots&\ddots&\ddots&\dots&\dots \\ 
\vdots&\dots&\ddots&\ddots&\dots \\
0&\dots&\dots&0&\mathrm {Id}
\end{pmatrix}\in \mathrm{GL}_{m}\Big(\mathcal{M}(\C^{*})\Big),$$ solution of 
$D\hat{H}^{\left[(\l_{i,j})\right]}=\sq \left(\hat{H}^{\left[(\l_{i,j})\right]}\right)\mathrm{Diag} (z^{n_{i}}C_{i})$, such that for all $i<j$, and for $|z|$ close to $0$, the matrix $\hat{H}_{i,j}^{\left[\l_{i,j}\right]}$ has simple poles contained in the spirals $-\l_{i,j} q^{(n_{j}-n_{i})^{-1}\Z}$. \par
Assume now that $k=2$  and let us define the finite set $\Sigma'\subset \C^{*}/q^{(n_{2}-n_{1})^{-1}\Z}$ as follows:
 $$\Sigma :=\left\{\l \in \C^{*}/q^{(n_{2}-n_{1})^{-1}\Z}\left|\exists \a_{1} \in \mathrm{Sp}( C_{1}),\a_{2} \in \mathrm{Sp}(C_{2}), \hbox{ such that } \l^{n_{2}-n_{1}}q^{-\frac{n_{2}-n_{1}-1}{2}}q^{\Z}=\frac{\a_{1}}{\a_{2}}q^{\Z}\right\} \right. .$$
The next theorem says that in this case, the above solution is exactly the same as the one in Theorem \ref{theo4}. 

\pagebreak[3]
\begin{theo}\label{theo3}
Let $\l\in \left(\C^{*}/q^{(n_{2}-n_{1})^{-1}\Z}\right)\setminus \Sigma$. Then, the entries of $\mathcal{\hat{B}}_{n_{2}-n_{1}}\left(\hat{H}_{1,2}\right)$ belong to $\mathbb{H}_{n_{2}-n_{1},n_{2}-n_{1}}^{[\l]}$ and 
$$S_{q}^{[\l]}\left(\hat{H}\right)=\hat{H}^{\left[\l\right]}:=\begin{pmatrix}
\mathrm {Id}&\hat{H}_{1,2}^{[\l]}\\
0&\mathrm {Id}
\end{pmatrix}.$$
\end{theo}

\begin{proof}[Proof of Theorem \ref{theo3}]
 Using $D=\hat{H}\Big[\mathrm{Diag} (z^{n_{i}}C_{i})\Big]_{\displaystyle\sq}$, we find,
\begin{equation}\label{eq5}
z^{n_{2}}\sq \left(\hat{H}_{1,2}\right)C_{2}=z^{n_{1}}C_{1}\hat{H}_{1,2}-U_{1,2}.
\end{equation}

Hence, there exist $K\in \N^{*}$, and a finite set $\Sigma '  \subset \C^{*}/q^{K^{-1}\Z}$, such that if ${\l\in \left(\C^{*}/q^{K^{-1}\Z}\right)\setminus\Sigma '}$, then the matrix  $$S_{q}^{[\l]}\left(\hat{H}\right):=\begin{pmatrix}
\mathrm {Id}&S_{q}^{[\l]}\left(\hat{H}_{1,2}\right)\\
0&\mathrm {Id}
\end{pmatrix}\in \mathrm{GL}_{m}\Big(\mathcal{M}(\C^{*})\Big),$$ is solution of 
$DS_{q}^{[\l]}\left(\hat{H}\right)=\sq \left(S_{q}^{[\l]}\left(\hat{H}\right)\right)\mathrm{Diag} (z^{n_{i}}C_{i})$, and $S_{q}^{[\l]}\left(\hat{H}_{1,2}\right)$ has simple poles that are contained in the $q^{1/K}$-spirals $$-\l q^{K^{-1}\Z}.$$
We claim that  $K=n_{2}-n_{1}$. Since the $n_{2}-n_{1}$ is an integer, we have to prove that each entry of $\hat{H}_{1,2}$ satisfies a linear $q$-difference equation with maximal slope $n_{2}-n_{1}$. This is a direct consequence of (\ref{eq5}) and the fact that the entries of $C_{1}$ and $C_{2}$ belong to $\C$. Hence $K=n_{2}-n_{1}$ and the properties described just above the theorem tell us that if ${\l\in \left(\C^{*}/q^{(n_{2}-n_{1})^{-1}\Z}\right)}$ does not belongs to neither $\Sigma '$ or $\Sigma$, then $$S_{q}^{[\l]}\left(\hat{H}\right)=\hat{H}^{\left[\l\right]}.$$
To conclude to proof, we have to show that $\Sigma '\subset \Sigma $. Let ${\l\in \left(\C^{*}/q^{(n_{2}-n_{1})^{-1}\Z}\right)\setminus \Sigma}$ and let us prove that $\l \notin \Sigma '$. Applying $\mathcal{\hat{B}}_{n_{2}-n_{1}}$ to $\hat{H}_{1,2}$, we find using Lemma \ref{lem1} that 
$$\z^{n_{2}-n_{1}}q^{-\frac{n_{2}-n_{1}-1}{2}}\mathcal{\hat{B}}_{n_{2}-n_{1}}\left(\hat{H}_{1,2}\right)C_{2}
=C_{1}\mathcal{\hat{B}}_{n_{2}-n_{1}}\left(\hat{H}_{1,2}\right)-\mathcal{\hat{B}}_{n_{2}-n_{1}}\left(z^{-n_{1}}U_{1,2}\right).$$
Due to Lemma \ref{lem3} below, the entries of $\mathcal{\hat{B}}_{n_{2}-n_{1}}\left(\hat{H}_{1,2}\right)$ belong to $\mathbb{H}_{n_{2}-n_{1},n_{2}-n_{1}}^{[\l]}$. Therefore, we have $\l \notin \Sigma '$. 
\end{proof}

\pagebreak[3]
\begin{lem}\label{lem3}
Let $\a\in \C$ be an eigenvalue of $X\mapsto C_{1}^{-1}XC_{2}$. Then, $\a$ is the quotient of an eigenvalue of $C_{2}$ by an eigenvalue of $C_{1}$. 
\end{lem}

\begin{proof}
 Let $\a\in \C$ that is not the quotient of an eigenvalue of $C_{2}$ by an eigenvalue of~$C_{1}$. Let us prove that $\a$ is not an eigenvalue of $X\mapsto C_{1}^{-1}XC_{2}$. This is equivalent to prove that $1$ is not an eigenvalue of $X\mapsto \a^{-1}C_{1}^{-1}XC_{2}$, which is equivalent to the fact that $X\mapsto \a C_{1}X-XC_{2}$ is not bijective. Let $X$ such that $\a C_{1}X=XC_{2}$. Then, for all $P\in \C[T]$, $P(\a C_{1})X=XP(C_{2})$. Taking $P=P_{1}$, the characteristic polynomial associated to $\a C_{1}$, we find that $0=XP_{1}(C_{2})$. The assumption we have made on $\a$ tell us that $\a C_{1}$ and $C_{2}$ have no eigenvalues in common. This implies that $P_{1}(C_{2})$ is invertible, and therefore $X=0$. We obtain that  $X\mapsto \a C_{1}X-XC_{2}$ is bijective. Hence, $\a$ is not an eigenvalue of $X\mapsto C_{1}^{-1}XC_{2}$.
\end{proof}

\fussy
\pagebreak[3]
\nocite{Ad,And,Be,Ca,DR,DSK,DVH,DV02,DV09,Ro11,Trj,Chak}
\bibliographystyle{alpha}
\bibliography{C:/Users/thomas.dreyfus/Dropbox/Maths/biblio}

\end{document}